\documentclass[notitlepage,a4,10.5pt]{article}

\usepackage{graphicx}

\usepackage{amsmath}%
\usepackage{amsfonts}%
\usepackage{amssymb}%
\usepackage{mathrsfs}
\usepackage{amsthm}
\usepackage{float}

\usepackage{authblk}

\usepackage[left=2.5cm,right=2.5cm,top=3cm,bottom=3cm]{geometry} 

\usepackage{xcolor}

\newcommand{\mc}{\mathcal}
\newcommand{\cp}{\times}

\newcommand{\bol}{\boldsymbol}

\newcommand{\abs}[1]{\left\lvert{#1}\right\rvert}
\newcommand{\w}{\wedge}
\newcommand{\lr}[1]{\left({#1}\right)}
\newcommand{\lrs}[1]{\left[{#1}\right]}
\newcommand{\lrc}[1]{\left\{{#1}\right\}}

\newcommand{\mf}{\mathfrak}
\newcommand{\p}{\partial}

\newcommand{\ti}[1]{\textit{#1}}
\newcommand{\tb}[1]{\textbf{#1}}

\theoremstyle{plain}
\newtheorem{theorem}{Theorem}[section]
\newtheorem{proposition}[theorem]{Proposition}
\newtheorem{corollary}[theorem]{Corollary}

\theoremstyle{definition}
\newtheorem{mydef}[theorem]{Definition}

\theoremstyle{remark}
\newtheorem{example}[theorem]{Example}
\newtheorem{remark}[theorem]{Remark}

\newcommand{\R}{\mathbb{R}}

\newcommand{\eq}[1]{\begin{equation}\begin{split}{#1}\end{split}\end{equation}}
\newcommand{\sys}[2]{\begin{subequations}\begin{align}{#1}\end{align}\label{#2}\end{subequations}}

\newcommand{\M}{M} 

\begin{document}

\title{
    Generalizing  Hamiltonian Mechanics with  Closed Differential Forms
}
\author[1]{Nathan Duignan} \author[2]{Naoki Sato}
\affil[1]{Department of Mathematics and Statistics, University of Sydney, \protect\\
    NSW, 2006, Australia \protect\\ Email: nathan.duignan@sydney.edu.au}
\affil[2]{National Institute for Fusion Science, \protect\\
    322-6 Oroshi-cho, Toki-city, Gifu 509-5292, Japan \protect\\ Email: sato.naoki@nifs.ac.jp}
\date{\today}
\setcounter{Maxaffil}{0}
\renewcommand\Affilfont{\itshape\small}

\maketitle
\begin{abstract}
    Classical Hamiltonian mechanics, characterized by a single conserved Hamiltonian (energy) and symplectic geometry, `hides' other invariants into symmetries of the Hamiltonian or into the kernel of the Poisson tensor.
    Nambu mechanics aims to generalize classical Hamiltonian mechanics to ideal dynamical systems bearing two Hamiltonians, but its connection to a suitable geometric framework has remained elusive.
    This work establishes a novel correspondence between generalized Hamiltonian mechanics, defined for systems with a phase space conservation law (invariance of a closed form) and a matter conservation law (invariance of multiple Hamiltonians), and multisymplectic geometry. The key lies in the invertibility of differential forms of degree higher than 2.
    We demonstrate that the cornerstone theorems of classical Hamiltonian mechanics (Lie-Darboux and Liouville) require reinterpretation within this new framework, reflecting the unique properties of invertibility in multisymplectic geometry.  Furthermore, we present two key theorems that solidify the connection: i) any classical Hamiltonian system with two or more invariants is also a generalized Hamiltonian system and ii) given a generalized Hamiltonian system with two or more invariants, there exists a corresponding classical Hamiltonian system on the level set of all but one invariant, with the remaining invariant playing the role of the Hamiltonian function.
\end{abstract}


\section{Introduction}
Classical Hamiltonian mechanics represents the
mathematical structure underlying the laws of physics.
It describes the dynamics of ideal (not subject to dissipation) systems
through conservation of energy (a property of matter encoded in the Hamiltonian $H$)
and, most importantly,
the conservation of the ambient space (the phase space characterized by
the symplectic $2$-form $\omega$ and the associated Liouville measure).
In its most general formulation, classical Hamiltonian mechanics arises as
a noncanonical Hamiltonian theory \cite{Morrison80,Morrison82} with a
Poisson bracket assigned by a contravariant Poisson $2$-tensor $\mc{J}$.
Remarkably, once the kernel of the Poisson tensor $\mc{J}$, spanned by the so-called Casimir invariants \cite{Littlejohn82,Morrison98}, is appropriately factored out,
a precise correspondence can be identified between $\omega$ and $\mc{J}$:
essentially, one is the \ti{inverse} of the other.
This property builds a duality between symplectic geometry and Poisson algebras,
which can be understood in terms of the Lie-Darboux \cite{Arnold89,DeLeon89} and Liouville theorems. These theorems establish that given an $n$-dimensional Poisson tensor $\mc{J}$
it is possible to identify local coordinates $\lr{p^1,...,p^m,q^1,...,q^m,C^1,...,C^s}$, with $n=2m+s$, such that the Lioville measure $dp^1\w...\w dp^m\w dq^1\w...\w dq^m\w dC^1\w...\w dC^s$ is invariant under the Hamiltonian flow, which in turn satisfies
canonical Hamiltonian equations on each level set of the Casimir invariants with symplectic $2$-form $\omega=\sum_{i=1}^m d p^i\w dq^i$.

The aim of generalized Hamiltonian mechanics
is to replace the $2$-dimensional building block $dp^i\w dq^i$ of classical phase space
with a higher $k$-dimensional building block with $k>2$, starting from the $k=3$ case $dp^i\w dq^i\w dr^i$ originally examined by Nambu \cite{Nambu}.
It is common to refer to the case $k=3$ as Nambu mechanics, and to the corresponding bracket generating dynamics by contraction with $k-1=2$ Hamiltonians as a Nambu bracket.
It should be emphasized that a higher dimensional building block
implies that one Hamiltonian is no longer sufficient to generate
the dynamics, because the contraction of a $k$-form with a vector field
produces a $k-1$-form. In particular, a generalized Hamiltonian theory of degree $k$ requires $k-1$ Hamiltonians.
Therefore, another way of thinking about generalized Hamiltonian mechanics is to understand it as the mathematical structure associated with the ideal dynamics of systems with more than one matter conservation law.
This interpretation is however somewhat naive because it hides the fact that
the key issue is how a higher number of Hamiltonians affects the phase space conservation law. In particular, we will see that in the type of generalized Hamiltonian mechanics discussed in the present work
the conservation of the symplectic $2$-form $\omega$ is replaced
by the conservation of a closed differential form of degree $k$.

One may wonder why do we need to study generalized Hamiltonian mechanics in the first place.
Indeed, classical Hamiltonian systems may possess other invariants in addition to the Hamiltonian, such as the Casimir invariants associated with the Poisson tensor $\mc{J}$, or any conservation law arising from a symmetry of the system as described by Noether's theorem. The matter of the fact is that, as it will be shown in the present study, a generalized Hamiltonian theory
of degree $k$ is not a subset of a classical Hamiltonian theory with $1$ Hamiltonian and $k-2$ additional conservation laws arising from symmetries of the Hamiltonian and/or from the kernel of the Poisson tensor. In other words, there are generalized Hamiltonian systems that cannot be described within the framework of classical Hamiltonian mechanics, and yet possess both phase space and matter conservation laws.

Another more practical reason to investigate generalized Hamiltonian mechanics is that
several physical systems often arise with multiple fundamental invariants.
In many cases there are $2$ invariants, the first representing the total energy of the system, while the second is usually associated with rotations (e.g. angular momentum in a rigid body or enstrophy in an ideal fluid) or exhibits a topological character (for example, helicity in an ideal fluid). In fact, the identification of Nambu brackets for fluid systems represents an active area of research \cite{Blender,Nevir,Fukumoto}.
In addition, several applications of Nambu mechanics have been proposed.
A non-exhaustive list includes M-theory \cite{Ho}, bracket quantization \cite{Yoshida,Cha,Awata,Dito,Bayen}, extension of Hamilton-Jacobi theory to Nambu mechanics \cite{Yon}, and formulation of metriplectic ternary brackets to describe dissipation \cite{Bloch13}.

Unfortunately, it is well known that the generalization of the notion of Poisson algebra arising in classical Hamiltonian mechanics to generalized Hamiltonian systems
is a non-trivial mathematical problem. The reason is that the Jacobi identity \cite{Olv}, which is the Poisson bracket axiom
guaranteeing the local existence of the symplectic $2$-form $\omega$,
does not admit a straightforward extension to the cases $k>2$.
Originally, a Lie-algebraic formulation of Nambu mechanics was first provided in
\cite{Bial91} which relied on an underlying Lie bracket bearing a noncanonical Hamiltonian structure. Later, the so-called fundamental identity has been proposed as replacement of the Jacobi identity \cite{Takhtajan1994}.
The fundamental identity has the merit that the bracket obtained by contracting
the Nambu bracket with one of the two Hamiltonians is a Poisson bracket, i.e. it satisfies the Jacobi identity.
However, the formulation of Nambu mechanics in terms of the fundamental identity suffers a limited range of applicability: even a fully antisymmetric constant contravariant $3$-tensor does not necessarily satisfy the fundamental identity. This is contrast with the Jacobi identity, which is always satisfied by any antisymmetric constant contravariant $2$-tensor.
As a consequence, there are classical Hamiltonian systems with $2$ invariants that cannot be described by a Nambu bracket endowed with the fundamental identity.
The axiomatic formulation of Nambu mechanics is discussed in detail in \cite{Sato21,Sato24}, where the existence of a closed $3$-form $w$ preserved by the dynamics is proposed as an alternative to the fundamental identity.
This formulation has the advantage of being less restrictive than the fundamental identity
(all constant contravariant $3$-tensors now qualify as Poisson $3$-tensors). However, the theory requires the subtle notion of invertibility of a differential form of degree $k$.

To avoid confusion, in the following we will refer to a
generalized Hamiltonian theory of degree $k$ in the sense of Definition \ref{def:generalisedHamiltonian} of Section \ref{sec:GeneralizedHamiltonianMechanics}.
Such theory is characterized by a phase space conservation law (the invariance of a closed differential form of degree $k$) and a matter conservation law (the invariance of $k-1$ Hamiltonian functions). When $k=3$, we may sometime talk about Nambu mechanics,
and refer to Nambu(FI) mechanics when dealing with the definition
of Nambu mechanics in terms of the fundamental identity.

Building on the notion of generalized Hamiltonian mechanics formulated in \cite{Sato21,Sato24}, the aim of this paper is to relate
a generalized Hamiltonian theory of degree $k>2$ to multisymplectic geometry (the geometry of closed differential forms) \cite{Ryvkin,Can} with the aid of the notion of invertibility of a differential form. As a result, we recover a duality between generalized Hamiltonian mechanics and multisymplectic geometry (note that here we use the term multisymplectic in a loose sense that does not explicitly include non-degeneracy conditions, but simply refers to the availability of a closed differential form).

The present paper is organized as follows. In Section \ref{sec:GeneralizedHamiltonianMechanics},
we define the notion of generalized Hamiltonian system.
In Section \ref{sec:classicalCorrespondance}, we prove two theorems
connecting a generalized Hamiltonian theory of degree $k$
with classical Hamiltonian mechanics with two or more invariants.
In particular, we find that i) any classical Hamiltonian system with two or more invariants is also a generalized Hamiltonian system and ii) given a generalized Hamiltonian system with two or more invariants
there exists a corresponding classical Hamiltonian system on the level set of all but one invariant, with the remaining invariant playing the role of the Hamiltonian function.
In Section \ref{sec:flatnessDarbouxEtc}, we explain how the cornerstone theorems of classical Hamiltonian mechanics (Lie-Darboux and Liouville) require reinterpretation within the new framework, reflecting the unique properties of invertibility in multisymplectic geometry, and prove
Lie-Darboux and Liouville type theorems within the generalized Hamiltonian setting.
Section \ref{sec:JacobiIdentity} discusses the role of the Jacobi identity in generalized Hamiltonian mechanics, its relationship with the fundamental identity, as well as the availability of invariant measures.
In Section \ref{sec:examples}, we provide three examples of generalized Hamiltonian systems.
In the first example, we show that quasisymmetric magnetic fields \cite{Rod}
can be described as a generalized Hamiltonian system of degree $k=3$ where
the Hamiltonians are given by the magnetic flux function and the magnetic field strength.
In the second example, we exhibit the case of a classical Hamiltonian system
with two invariants that can be cast as a generalized Hamiltonian system of degree $k=3$, but yet the corresponding Nambu bracket fails to satisfy the fundamental identity.
The third example concerns a $4$-dimensional dynamical system governed by an antisymmetric contravariant $2$-tensor failing to satisfy the Jacobi identity, and yet bearing a $4$-dimensional generalized Hamiltonian structure of degree $k=4$. We also show how the results of Theorem \ref{thm:classicalFromGeneralised} below can be applied to recover a classical Hamiltonian system on a $2$-dimensional submanifold of the original $4$-dimensional domain.
Concluding remarks are given in Section \ref{sec:ConcludingRemarks}.

\subsection{Conventions on pairing and contraction of forms and multi-vectors}

Here we summarize the notation and conventions adopted in this paper.
Einstein summation convention on repeated indices will be used.
Moreover, when there is no ambiguity, a lower index will be used to specify partial derivatives. For example, given a contravariant $2$-tensor
with components
$\mc{J}^{ij}$ we will often write $\mc{J}^{ij}_m=\p \mc{J}^{ij}/\p x^m$, where $x^m$ denotes the $m^{\text{th}}$ coordinate.

For any vector space $V$ denote by $\Lambda^k(V)$ the set of $k$-vectors and by $\Lambda^k(V^*)$ the set of $k$-forms over $V$. For any smooth manifold $\M$ of dimension $n$, denote by $\Omega^k(\M)$ the set of smooth differential $k$-forms on $\M$, by $\mc{X}^k(\M)$ the set of smooth $k$-vector fields on $\M$, and by $C^\infty(\M)$ the set of smooth functions on $\M$.

For any vector field $X\in \mc{X}(M)$, denote the interior product by
\[\iota_X: \Omega^k(M) \to \Omega^{k-1}(M),\qquad \alpha \mapsto \iota_X \alpha.\]
Further, we extend the interior product for any $J \in \mc{X}^m(M)$ so that, for decomposable elements which have the form $J = J_1\w\dots\w J_m$, $J_i \in \mc{X}(M)$, the interior product is given by
\[ \iota_J = \iota_{J_m}\dots\iota_{J_1} :\Omega^k(M) \to \Omega^{k-m}(\Omega).\]
The interior product is then defined for all $J\in\mc{X}^m(M)$ by linearity. Similarly, for any $\alpha\in\Omega^1(M)$ denote the interior product by $\iota_\alpha: \mc{X}^k(M) \to \mc{X}^{k-1}(M)$, and for any $w \in \Omega^m(\M)$ define the interior product for decomposable elements which have the form $w = \alpha^1 \w \dots \w \alpha^m$ as $\iota_w = \iota_{\alpha^m}\dots\iota_{\alpha^1}:\mc{X}^k(\M)\to\mc{X}^{k-m}(\M)$, and extend to all elements by linearity.

If $p\in \M$ and $U$ is a local chart of $p$, we will denote by $(x^1,\dots,x^n):U \to \R^n $ as the coordinates of the chart and let $\lr{\p_1,...,\p_n}$ and $\lr{dx^1,...,dx^n}$ denote a basis of $\mc{X}(U)$ and dual basis of $\Omega(U)$, respectively. We will use the shorthand notation
\[\partial_{i_1\dots i_m} = \partial_{i_1}\w{...\w}\p_{i_m} \qquad {dx^{i_1...i_m}}=dx^{i_1}\w{\dots\w} dx^{i_m}.\]

With the above notation, the following local constructions of the interior product and natural pairing between $\Omega^k(M)$ and $\mc{X}^k(M)$ can be inferred.
\begin{itemize}
    \item \ti{Pairing}: Extend the pairing $\langle \p_i,dx^j\rangle=\delta_i^j$ at each point $p\in U$ to basis elements of $\bigwedge^k T_p\M$ and $\bigwedge^k T^{\ast}_p\M$ according to
          \eq{
              \langle \p_{i_1}\w ... \w \p_{i_k}, dx^{j_1}\w ...\w dx^{j_k} \rangle=\delta^{i_1...i_k}_{j_1...j_k},~~~~i_1<i_2<...<i_k,~~~~j_1<j_2<...<j_k,\label{contr}
          }
          where $\delta^{i_1...i_k}_{j_1...j_k}=\delta_{j_1}^{i_1}...\delta^{i_k}_{j_k}$. Observe that, for a given $\alpha\in\Omega^{k}\lr{\M}$ and $J\in\mc{X}^k\lr{\M}$, it holds that $\iota_{J}\alpha= \langle \alpha, J \rangle =: \alpha\lr{J}$.
    \item \ti{Contraction}: As an example of local computation of contraction, take $\mc{J}\in\mc{X}^2(U)$ and $dH\in \Omega(U)$. The (left) contraction between $\mc{J}$ and $dH$ can be evaluated as
          \eq{
          \iota_{dH}\mc{J}=\iota_{dH}\sum_{i<j}\mc{J}^{ij}\p_{ij}
          =-\mc{J}^{ij}H_j\p_j,
          }
          where $H_j=\p H/\p x^j$.
          Similarly, take $J\in\mc{X}^3(U)$ and $dC\in \Omega(U)$. 
          \eq{
          \iota_{dH}\iota_{dC}J
          =\iota_{dH}\sum_{j<k}J^{ijk}C_i\p_{jk}=-J^{ijk}H_jC_k\p_i.
          }
\end{itemize}
The contraction between forms and multivectors are (locally) can be locally defined through the pairing \eqref{contr}. For example, for $m<n$ and $i<j<k$ we have
\eq{
    \iota_{dx^{mn}}\p_{ijk}=\langle dx^{mn},\p_{ij}\rangle\p_k+\langle dx^{mn},\p_{jk}\rangle\p_i-\langle dx^{mn},\p_{ik}\rangle\p_j=\delta_{ij}^{mn}\p_k+\delta_{jk}^{mn}\p_i-\delta_{ik}^{mn}\p_j.
}
Similarly
\eq{
\iota_{dC\w dH}J=&\sum_{m<n,i<j<k}J^{ijk}\lr{C_mH_n-C_nH_m}\iota_{dx^{mn}}\p_{ijk}\\
=&
\sum_{m<n,i<j<k}J^{ijk}\lr{C_mH_n-C_nH_m}\lr{\delta_{ij}^{mn}\p_k+\delta_{jk}^{mn}\p_i-\delta_{ik}^{mn}\p_j}\\
=&\sum_{i<j<k}J^{ijk}\lr{C_iH_j-C_jH_i}\p_k+J^{ijk}\lr{C_jH_k-C_kH_j}\p_i+J^{ijk}\lr{C_kH_i-C_iH_k}\p_j\\=&
-\sum_{k<j}J^{ijk}\lr{H_jC_k-C_jH_k}\p_i\\=&-J^{ijk}H_jC_k\p_i\\=&\iota_{dH}\iota_{dC}J.
}

\section{Generalized Hamiltonian Mechanics}
\label{sec:GeneralizedHamiltonianMechanics}

The aim of this section is to define the framework of generalized Hamiltonian mechanics.
In the following, we shall assume smoothness of the involved quantities.
Given a manifold $M$, a vector field $X\in \mc{X}(M)$, and coordinates $\bol{x}=\lr{x^1,...,x^n}$ in a neighborhood $U$ of a point $p\in M$, we consider a dynamical system governed by the equations of motion
\eq{
    \frac{dx^i}{dt}=X^i,\quad i=1,\dots, n\label{EoM}
}
where $t$ denotes the time variable.

\subsection{Definition of some generalized Hamiltonian systems}

The dynamical system \eqref{EoM} qualifies as a generalized Hamiltonian system according to the following definition:

\begin{mydef} \label{def:generalisedHamiltonian}
    Let $\M$ be a smooth manifold of dimension $n\geq k\geq 2$. A vector field $X\in\mc{X}(\M)$ is called a \emph{generalized Hamiltonian vector field of degree $k$} whenever there exists a closed differential form $w \in \Omega^k(M)$ and $k-1$ almost everywhere linearly independent
    exact $1$-forms $dH^i$, $i=1,...,k-1$, such that
    \begin{equation}
        \iota_Xw=-dH^1\w ...\w dH^{k-1}.
        \label{ghmk}
    \end{equation}
    We call the $H^i$ the Hamiltonian functions.
\end{mydef}

Equation \eqref{ghmk} is referred to as the `Hamilton-de Donder-Weyl' (HDW) equation by some authors \cite{Ryvkin}.
Recalling Cartan's homotopy formula for the Lie derivative $\mc{L}_X$ and using the closure of $w$,  equation \eqref{ghmk} implies
\begin{equation}
    \mc{L}_Xw= d\iota_X w + \iota_X dw = 0.\label{Lw}
\end{equation}
Equation \eqref{Lw} represents a \textit{phase space conservation law}, indicating that it conserves  the geometric structure of $M$ independently of the specific form of the Hamiltonian functions $H^i$, $i=1,...,k-1$, which in turn encode the physical properties of matter.

A $2$-form is non-degenerate if, when considered as an antisymmetric matrix at any point $p\in M$, it is invertible. When $w$ is a non-degenerate $2$-form , $w$ is a symplectic form and the classical Hamiltonian structure is recovered.
Similar to the classical Hamiltonian structure, in the generalised case, the Hamiltonians $H^i$ usually correspond to the energy of the system and additional conservation laws that originate, for example, from symmetries. See Section \ref{sec:examples} for some examples.
The conservation of the Hamiltonians can be deduced from \eqref{ghmk} and antisymmetry
by contracting the $k$-form $w$ with $X$ twice:
\begin{equation}
    0=\iota_X^{2}w=-\lr{\iota_XdH^1}dH^2\w ...\w dH^{k-1}-...-\lr{-1}^{k-2}\lr{\iota_XdH^{k-1}}dH^1\w ...\w dH^{k-2}.
\end{equation}
Since by definition the $1$-forms $dH^i$, $i=1,...,k-1$, are linearly independent almost everywhere, then $dH^1\w ...\w dH^{k-1}\neq 0$ almost everywhere. By the assumed smoothness of $X$ and the $H^i$, it follows that $\iota_XdH^i=0$ for all $i=1,...,k-1$.

More generally, if a vector field $X$ a priori only preserves the closed $k$-form $w$ then
\[ \mc{L}_X w = 0 \iff d\iota_X w = 0.\]
It follows that $\iota_X w = -\sigma$ where $\sigma \in \Omega^{k-1}(M)$ is closed. In general, there may not exist Hamiltonian functions $H^i,\, i=1,\dots, k-1$ such that $\sigma = dH^1\w\dots\w dH^{k-1}$. Hence, although $X$ preserves $w$, it may not be a generalised Hamiltonian vector field. This motivates the following definition.

\begin{mydef}
    Let $\M$ be a smooth manifold of dimension $n\geq k\geq 2$ and $w$ a closed $k$-form on $\M$. A vector field $X\in\mc{X}(\M)$ is called a
    \emph{HDW (Hamilton-de Donder-Weyl)
        vector field} whenever there exist a closed $(k-1)$-form $\sigma$ on $\M$ such that
    \begin{equation}
        \iota_X w=-\sigma.
        \label{ghmk_2}
    \end{equation}
    The $k-1$ form $\sigma$ is the associated \emph{Hamiltonian $k-1$-form}.
\end{mydef}

Note that $\iota_X\sigma=-\iota_X^2w=0$ and $d\sigma=0$ imply
\begin{equation}
    \mc{L}_X\sigma=0
\end{equation}
Physically, the conservation of $\sigma$ may be regarded as a {matter conservation law}, in contrast with the conservation of $w$. Equation \eqref{ghmk_2} thus relates the properties of space $w$ to the properties of matter $\sigma$ through the vector field $X$.
These observations suggest that an even more general class of dynamical systems could be introduced, which does not require the forms \( w \) and \( \sigma \) to be closed. Such systems are governed by the equations
\begin{equation}
    \mathcal{L}_X w = 0, \quad \mathcal{L}_X \sigma = 0. \label{ghmk_3}
\end{equation}
In this framework, the phase space \( k \)-form \( w \) and the Hamiltonian \( k-1 \)-form \( \sigma \) are not necessarily closed differential forms.
In this more general class of dynamical systems, space and matter are not so obviously related as there is no immediate relation $\iota_X w = - \sigma$. In the following, we shall be concerned only with generalized Hamiltonian systems as defined in Definition \ref{def:generalisedHamiltonian}.

\subsection{Invertibility and Poisson tensors}

When $w$ is a non-degenerate $2$-form, equation \eqref{ghmk} uniquely determines the vector field $X$ for a given Hamiltonian function $H^1$.
However, when $k>2$,
a dimensional obstruction arises to the solvability of
equation \eqref{ghmk} for $X$. The obstruction depends upon the $k$-form $w$ and the Hamiltonians $H^1,...,H^{k-1}$.
Indeed, consider the hat-map
\[\hat{w}:\mc{X}\lr{\M}\rightarrow\Omega^{k-1}\lr{\M},\quad \hat{w}(X) = \iota_X w,\]
which sends vector fields into $k-1$-forms. The map $\hat{w}$ is linear and can be surjective at a point $p\in \M$ only if ${\rm dim}\lr{\mc{X}\lr{M}}\geq {\rm dim}\lr{\Omega^{k-1}\lr{M}}$. Explicitly, the dimensions of the spaces give the necessary condition for surjectivity of $\hat{w}$ as,
\begin{equation}
    n\geq \binom{n}{k-1}.\label{sur}
\end{equation}
When \eqref{sur} is violated, there exist $k-1$ forms $\sigma\in\Omega^{k-1}\lr{\M}$ with no generating vector field $X$ such that $\iota_Xw=-\sigma$.
For example, for $k=3$, eq. \eqref{sur} gives the condition $n\leq 3$, while $k=2$ returns $n\geq n$. This behavior is the first nontrivial property of equation \eqref{ghmk} that separates classical
Hamiltonian mechanics ($k=2$) from  generalized Hamiltonian systems  ($k>2$).

If ${\rm ker}\lr{\hat{w}}:=\lrc{Y\in \mc{X}\lr{\M}:\iota_Yw=0}\neq\emptyset$ then equation \eqref{ghmk} does not determine $X$ uniquely. Indeed, $X+Y$ also satisfies eq. \eqref{ghmk} for any $Y\in{\rm ker}\lr{\hat{w}}$.
Thus, the phase space conservation law and the conservation of the $k-1$ Hamiltonians do not provide enough information to determine the equations of motion $X$. It is here that the notions of invertibility of differential forms and Poisson $k$-tensor  come into play.
In this study we will be concerned with the following notions of invertibility for $k$-forms and $k$-vectors:
\begin{mydef}
    Let $w\in\Omega^k\lr{M}$ and $J\in\mc{X}^k\lr{M}$. We say that $J$ is a (strong) inverse of $w$ (and vice versa) if
    \eq{
        \iota_{\iota_Xw}J=X,\qquad\forall X\in \mc{X}\lr{\M}.\label{sinv}
    }
\end{mydef}

\begin{mydef}
    Let $w\in\Omega^k\lr{M}$ and $J\in\mc{X}^k\lr{M}$.
    Consider a smooth $m$-dimensional distribution $\Delta$ on $M$.
    We say that $J$ is a $\Delta$-inverse of $w$ (and vice versa) if
    \eq{
        \iota_{\iota_X w}J=X~~~~\forall X\in \mc{X}_\Delta(M), \label{winv}
    }
    where $\mc{X}_\Delta(M)$ is the space of all smooth sections of $\Delta$.
\end{mydef}
When $k=2$, in local coordinates, the (strong) invertibility \eqref{sinv} of the symplectic $2$-form $\omega=\sum_{i<j}\omega_{ij}dx^{ij}$ corresponds to the invertibility of the covariant $2$-tensor $\omega_{ij}$.
The inverse $\mc{J}=\sum_{i<j}\mc{J}^{ij}\p_{ij}\in\mc{X}^2\lr{\M}$ such that $\omega_{ij}\mc{J}^{jk}=\delta_j^k$ is the Poisson $2$-tensor.
Furthermore, the  closure of $\omega$
translates into the Jacobi identity satisfied by $\mc{J}$. The relationship between the properties of $\omega$ and those of $\mc{J}$ are discussed in detail in \cite{Sato21}.
When $\omega$ is invertible, $X=-\iota_{dH}\mc{J}=\mc{J}^{ij}H_j\p_i$ uniquely, with $H=H^1$.

More generally, the condition for invertibility \eqref{sinv} for a $k$-form $w$ is stated in a given chart as
\begin{equation}
    \sum_{i_1<...<i_{k-1}}w_{ai_1...i_{k-1}}J^{i_1...i_{k-1}b}=
    \delta_a^b,
\end{equation}
where $J\in\mc{X}^k\lr{\M}$ is a $k$-vector with components $J^{i_1...i_k}$ of the corresponding $k$-times contravariant antisymmetric tensor. A necessary condition for the inverse \eqref{sinv} to exist is that the tensor $w_{i_1...i_k}$ has maximum rank $n$. This necessary condition is equivalent to the hat-map $\hat{w}$ being injective. An alternative computation of the rank is given in \cite{Sato24}. There, the rank of $w_{i_1...i_k}$ is locally computed as the number of linearly independent columns in the $n\times n^{k-1}$ matrix  $w_{i_1\lr{i_2...i_{k}}}$, where $i_1$ and  $\lr{i_2...i_k}$ identify rows and columns respectively. Indeed, when $w$ has rank $n$,
the matrix $w_{i_1\lr{i_2...i_{k}}}$ admits a right-inverse $J^{\lr{i_1...i_{k-1}}i_k}$ of dimension $n^{k-1}\times n$, although this tensor is not necessarily antisymmetric.

If an inverse $J$ of $w$ exists, then it can be used to compute the solution to $\iota_X w = -\sigma$. Using local coordinates and setting
\[ \sigma=dH^1\w...\w dH^{k-1}=\sum_{i_1<...<i_{k-1}}\sigma_{i_1...i_{k-1}}dx^{i_1...i_{k-1}},\]
the equation $\iota_X w = -\sigma$ becomes
\begin{equation}
    X^{j}w_{ji_1...i_{k-1}}
    =-\sigma_{i_1...i_{k-1}}
    ,~~~~i_1,...,i_{k-1}=1,...,n.\label{ghmk2}
\end{equation}
Suppose that, given $w$ and $H^1,...,H^{k-1}$, a solution $X$ of equation \eqref{ghmk2} exists.
If such $J$ exists, and for a given $\sigma, w$ there is a solution $X$ to eq.~\eqref{ghmk2}, then $X$ is given locally as
\begin{equation}
    X^j=-\sum_{i_1<...<i_{k-1}}J^{i_1...i_{k-1}j}\sigma_{i_1...i_{k-1}},~~~~j=1,...,n.
\end{equation}

As it will be clear from Section \ref{sec:classicalCorrespondance}, the notion of invertibility \eqref{sinv} is strong. I will be seen that there are cases of $k$-forms $w$ which do not admit an inverse, yet, a $k$-vector $J$ and $(k-1)$-form $\sigma = dH^1 \w\dots \w dH^{k-1}$ exist on $M$ such that there is a solution $X$ of $\iota_X w = -\sigma$ and the solution can be written as $X=-\iota_{\sigma}J$.
In particular, since the vector field $X$ solving $\iota_X w = -\sigma$ preserves all Hamiltonian functions $H^1,...,H^{k-1}$, it is sufficient to consider the weaker notion of $\Delta$-invertibility of the $k$-form $w$ with the distribution $\Delta$ defined at a point $p\in M$ by
\[ \Delta_p=\lrc{Y\in T_p\M\, :\, \iota_Y (dH^i|_p)=0,\, i=2,...,k-1}.\]
That is, any smooth section of $\Delta$, say $Y \in \mc{X}_\Delta$, is a vector field tangent to the $n-k+2$ dimensional submanifolds
\[\Sigma_h=\lrc{p\in{\M}\, :\, H^2\lr{p}=h^2\in\mathbb{R},\dots,H^{k-1}\lr{p}=h^{k-1}\in\mathbb{R}}.\]

In this study, we adopt the following definition of Poisson $k$-tensor $J$, leaving the discussion of its relationship with the inverse of the symplectic form $w$ to Sections 3 and 5.

\begin{mydef} \label{def:PoissonkTensor}
    A multivector field $J\in\mc{X}^k\lr{\M}$ is called a Poisson tensor of degree $k$ (or simply Poisson $k$-tensor)  whenever there exist $k-2$ linearly independent exact
    $1$-forms
    $dC^1,...,dC^{k-2}$
    such that the bivector $\mc{J}\in\mc{X}^2\lr{\M}$ defined by
    \eq{
        \mc{J}=\iota_{dC^1}....\iota_{dC^{k-2}}J,
    }
    is a Poisson $2$-tensor in the standard sense.
\end{mydef}

For any Poisson $k$-tensor, the functions $C^1,...,C^{k-2}$ are Casimir invariants of corresponding Poisson $2$-tensor $\mc{J}$.
Let $dH=dH^1$ and $dH^i=dC^{i-1}$, $i=2,...,k-1$. Then, the equations of motion associated with a Poisson $k$-tensor  $J$ take the form
\begin{equation}
    X=-\iota_{dH}\mc{J}=-\iota_{dH^1}...\iota_{dH^{k-1}}J.\label{gpt}
\end{equation}
When $k=2$, equation \eqref{gpt} locally gives $X=J^{ij}H_j\p_i$. When $k=3$, equation \eqref{gpt} gives $X=J^{ijk}H_jC_k\p_i$ with $H^1=H$ and $H^2=C$.
Due to antisymmetry of $J^{i_1...i_k}$, equation \eqref{gpt} preserves the Hamiltonians $H^1,...,H^{k-1}$.

\section{Correspondence between classical Hamiltonian mechanics and generalized Hamiltonian mechanics of degree $k$}
\label{sec:classicalCorrespondance}

In this section, we establish several results connecting classical Hamiltonian mechanics to generalized Hamiltonian mechanics of degree \( k \)
for a special class of closed $k$-forms $w$. We begin with the following statement that demonstrates a large class of generalized Hamiltonian systems that are derived from classical Hamiltonian systems endowed with multiple invariants.

\begin{theorem}\label{thm:wFromOmega}
    Let $\omega\in\Omega^2(\M)$, consider the functions $ H, C^1,\dots, C^{k-2}\in C^{\infty}(M)$, and define $dC := dC^1\w\dots\w dC^{k-2} \in \Omega^{k-2}(M)$. If there exists a vector field $X$ on $\M$ such that
    \begin{equation}
        d\omega\w dC=0,~~~~\lr{\iota_{X}\omega+dH}\w dC=0,~~~~\iota_XdC=0,
    \end{equation}
    then, the $k$-form
    \begin{equation}
        w=\omega\w dC,
    \end{equation}
    satisfies
    \begin{equation}
        \iota_X w=-dH\w dC,~~~~dw=0.
    \end{equation}
    That is, $X$ is a generalised Hamiltonian vector field of $w$ with Hamiltonians $H,C^1,\dots,C^{k-2}$.
\end{theorem}
\begin{proof}
    Observe that
    \[ \iota_X w = (\iota_X\omega)\w dC +\omega\w \iota_X dC = -dH\w dC, \]
    and
    \[ dw = d\omega\w dC = 0,\]
    as desired.
\end{proof}

Theorem \ref{thm:wFromOmega} implies that given a classical Hamiltonian system $\iota_X\omega=-dH$, $d\omega=0$,  and $k-2$ invariants $C^1,...,C^{k-2}$, there always exists a corresponding  generalized Hamiltonian system of degree $k$.
An analogous   result can be obtained in terms of $k$-vector fields.

\begin{proposition}\label{prop:PoissonkTensorFromPoisson2Tensor}
    Let $\mc{J}\in\mc{X}^2\lr{\M}$ denote a Poisson $2$-tensor with $k-2$ Casimir invariants  $C^1,...,C^{k-2}\in C^{\infty}\lr{\M}$, $dC^i\in{\rm ker}\lr{\mc{J}}$, $i=1,...,k-2$, such that the $1$-forms $dC^1,...,dC^{k-2}$ are linearly independent. Then,
    \eq{
        J=N\w\mc{J},
    }
    is a Poisson $k$-tensor  with $N=n_{k-2}\w...\w n_{1}\in\mc{X}^{k-2}\lr{\M}$, $n_i\in\mc{X}\lr{\M}$, $\iota_{n_i}dC^j=\delta_i^j$, $i,j=1,...,k-2$.
\end{proposition}
\begin{proof}
    Choose vector fields $n_1,...,n_{k-2}\in\mc{X}\lr{\M}$ 
    such that
    $\iota_{n_i}dC^j=\iota_{dC^j}n_i=\delta_i^j$, $i,j=1,...,k-2$. We have
    \eq{
        \iota_{dC^1}...\iota_{dC^{k-2}}J=\mc{J},
    }
    where we used the fact that the $C^i$, $i=1,...,k-2$,  are Casimir invariants of the Poisson $2$-tensor $\mc{J}$.
    Hence, from Definition \ref{def:PoissonkTensor} we see that $J$ is a Poisson $k$-tensor.
\end{proof}

Proposition \ref{prop:PoissonkTensorFromPoisson2Tensor} implies that given a classical Hamiltonian system with Poisson $2$-tensor $\mc{J}$, Hamiltonian $H$, and Casimir invariants $C^1,...,C^{k-2}$, there always exist a corresponding generalized Hamiltonian system of degree $k$ with Poisson $k$-tensor  $J=N\w\mc{J}$ and $k-1$ Hamiltonian functions  $H,C^1,...,C^{k-2}$ such that
the equations of motion take the form
$X=-\iota_{dH}\mc{J}=-\iota_{dH}\iota_{dC^1}...\iota_{dC^{k-2}}J$.

Remarkably,  the results above
admit converse statements;
given a generalized Hamiltonian system of degree $k$ it is possible to recover a classical Hamiltonian system under suitable assumptions on $w$ or $J$.
To see this, we introduce the following notation. Let $C^1,\dots, C^m \in C^{\infty}(\M)$ and define $dC := dC^1\w\dots\w dC^m$. Moreover, let
\[\operatorname{Reg}(C) := \{ p\in \M\,:\, dC|_p := (dC^1\w\dots\w dC^m)|_p \neq 0 \} \]
be the set of regular points of $dC$. Finally, denote the regular level sets by
\[\Sigma_c := \{ p\in\operatorname{Reg}(C)\,:\,C(p) = c\in\R^{k-2}\} \]
and set $i_c: \Sigma_c \hookrightarrow \M$ as the inclusion operator.
\begin{theorem}\label{thm:classicalFromGeneralised}
    Let $w\in\Omega^k\lr{\M}$ be a closed $k$-form
    and consider $C^1,\dots, C^{k-2} \in C^{\infty}(\M)$. Suppose that for every smooth function $H\in C^{\infty}(\M)$ there exists a vector field $X_{H}\in\mc{X}\lr{\M}$
    such that
    \eq{\iota_{X_H} w = - dH\w dC.}
    Then,
    there exists a 2-form $\omega\in \Omega^2\lr{\operatorname{Reg}(C)}$ such that
    \eq{w = \omega\w dC,\qquad d\omega\w dC = 0~~~~{{\rm in}~~\operatorname{Reg}\lr{C}}.}
    In particular, any regular level set $\Sigma_c$ of $C$ is a symplectic manifold with symplectic form $\tilde{\omega} := i_c^*\omega$ and
    \eq{\iota_{\tilde{X}_H}\tilde{\omega} = -d \tilde{H}, \label{oinv}}
    where $\tilde{X}_H$ is the restriction of $X_H$ to $\Sigma_c$ and $\tilde{H} := i_c^* H$.
\end{theorem}
\begin{proof}
    Let $n_1,\dots, n_{k-2}$ be a set of vector fields on $\operatorname{Reg}(C)$ such that $\iota_{n_i} dC^j = \delta_{i}^{j}$. Define
    \eq{\omega = \iota_{n_{k-2}}\dots \iota_{n_1} w = \iota_{N} w}
    where $N = n_1\w\dots\w n_{k-2}$. Then, for any $X_H$ satisfying $\iota_{X_H} w = -dH\wedge dC$ we have
    \eq{ \iota_{X_H} \omega = \iota_{X_H}\iota_{N} w = (-1)^{k-2} \iota_{N} \iota_{X_H} w =  (-1)^{k-1}\iota_{N} (dH\w dC) = - dH + (\iota_{n_i}dH) dC^i.\label{XHomega}}
    Applying $i_c^*$ to both sides of Equation \eqref{XHomega} yields the desired Equation \eqref{oinv}, namely, $\iota_{\tilde{X}_H}\tilde{\omega} = -d\tilde{H}$. The restriction of $X_H$ to $\Sigma_c$ is well-defined as $\iota_{X_H} dC_i = 0$ for $i=1,\dots, k-2$.

    We now show that $w = \omega\w dC$ and $d\omega\w dC = 0$ on $\operatorname{Reg}(C)$. Fix $p\in \operatorname{Reg}(C)$ and let $H^1,\dots, H^{n-k+2}$ be functions on $\M$ such that $(dH^1\w \dots \w dH^{n-k+2}\w dC^1\w \dots\w dC^{k-2})|_{p} \neq 0$. Such functions can always be found in a local chart about $p$ {and extended smoothly to the entirety of $\M$}.
    In particular, we may take local coordinates  $\lr{x^1,...,x^n}=\lr{H^1,...,H^{n-k+2},C^1,...,C^{k-2}}$. Recalling \eqref{XHomega}, we see that
    \eq{
        \iota_{X_H}\lr{\omega\w dC-w}=0\iff \omega\w dC-w=\xi,
    }
    where $\xi\in\Omega^k\lr{\operatorname{Reg}(C)}$ is a $k$-form such that $\iota_{X_H}\xi=0$.
    Due to the arbitrariness of $H$, equation \eqref{oinv}
    implies that $\tilde{\omega}$ is non-degenerate and thus invertible on $\Sigma_c$ with inverse $\tilde{\mc{J}}\in\mc{X}^2\lr{\Sigma_c}$ (this also implies that $n-k$ must be even). Hence, the equation $\iota_{X_H}\xi=0$ can be written in components as
    \eq{
    X_H^i\xi_{ij_1...j_{k-1}}=-\tilde{H}_m\tilde{\mc{J}}^{mi}\xi_{ij_1...j_{k-1}}=0,~~~~j_1,...,j_{k-1}=1,...,n.
    }
    However, the invertibility of $\tilde{\omega}$ on $\Sigma_c$ implies that
    $\tilde{\mc{J}}$ has full rank on $\Sigma_c$.
    Combined with the arbitrariness of $H$, it follows that
    $\xi_{ij_1,...,j_{k-1}}=0$ for all $i=1,...,n-k+2$ and $j_1,...,j_{k-1}=1,...,n$. Now suppose that $\xi_{ij_1,...,j_{k-1}}\neq 0$ for some $i$ between $n-k+3$ and $n$.
    Since $\xi$ is a $k$-form, at least $1$ of the indexes, say $j_1$,  belongs to $1,...,n-k+2$. On the other hand, antisymmetry implies that $\xi_{ij_1...j_{k-1}}=-\xi_{j_1ij_2....j_{k-1}}=0$.

    We have thus shown that
    \eq{
    w=\omega\w dC~~~~{\rm in}~~\operatorname{Reg}(C).
    }
    This also implies that $d\omega \w d C = dw = 0$ as desired. We now show that the non-degenerate $2$-form $\tilde{\omega} = i_c^*\omega$ is a symplectic form on each $\Sigma_c$. 
    First, from the fact that $d\omega\wedge dC = 0$ we have that
    \[ 0 = \iota_{N}(d\omega\w dC) = (-1)^{k-2} d\omega + \sum_{k}\alpha_k\w dC^k \]
    where $\alpha_k$ are 2-forms. Then $d\tilde{\omega} = d i_c^*\omega = i_c^*  d\omega = 0$ so that $\tilde{\omega}$ is closed.
    This concludes the proof.
\end{proof}

The following is a corollary of Theorem \ref{thm:classicalFromGeneralised}.

\begin{corollary}
    Consider the hypothesis of Theorem \ref{thm:classicalFromGeneralised}. Then, there exists a $k$-vector  $J\in\mc{X}^k\lr{\operatorname{Reg}(C)}$ such that
    \eq{
        X_H= -\iota_{dH \w dC }J=-\iota_{dH}\mc{J},
    }
    where $\mc{J}\in\mc{X}^2\lr{\operatorname{Reg}(C)}$. Furthermore, the restriction  $\tilde{\mc{J}}\in\mc{X}^2\lr{\Sigma_c}$ of $\mc{J}$ to $\mc{X}^2\lr{\Sigma_c}$
    is a Poisson $2$-tensor, and, up to sign, $J$ is a $\Delta$-inverse of $w$ with $\Delta_p =\lrc{Y\in T_p\lr{\operatorname{Reg}(C)}\,:\,\iota_Y (dC|_p)=0}$.
\end{corollary}

\begin{proof}
    By Theorem \ref{thm:classicalFromGeneralised}, the $2$-forms $\tilde{\omega}_c :=i^{\ast}_c\omega$ are closed and have full rank on the corresponding regular level set $\Sigma_c$. Consequently, each $\tilde{\omega}_c$ admits an inverse on the corresponding $\Sigma_c$, that is, a Poisson $2$-tensor $\tilde{\mc{J}}_c\in\mc{X}^2\lr{\Sigma_c}$ such that $\iota_{\iota_Y \tilde{\omega}_c}\tilde{\mc{J}}_c = Y$ for any $Y \in \mc{X}(\Sigma_c)$.

    Due to the smoothness of $\omega$ in $\operatorname{Reg}(C)$, the Poisson $2$-tensors $\tilde{\mc{J}}_c$ are smooth in $c$ and so there exists a $2$-tensor $\mc{J}\in \mc{X}^2\lr{\operatorname{Reg}(C)}$ such that, for any $\Sigma_c$, we have $ (i_c)_* \tilde{\mc{J}}|_p = \mc{J}|_p $ for all $p \in \Sigma_c$. In particular, there are local coordinates $\lr{x^1,...,x^n}=\lr{x^1,...,x^{2m},C^1,...,C^{k-2}}$ so that $\mc{J}=\sum_{i<j}^{2m}\tilde{\mc{J}}^{ij}\p_i\w\p_j$. It follows that, for any
    $Y \in \mc{X}(\operatorname{Reg}(C))$ everywhere tangent to $\Sigma_c$, we have
    \eq{ \iota_{\iota_Y \omega} \mc{J} = Y^i \omega_{ij} \mc{J}^{jm} \partial_m = Y^i \tilde{\omega}_{ij} \tilde{\mc{J}}^{jm} = Y^i \partial_i.\label{eq:mcJisDeltaInverse}}
    Hence, $\mc{J}$ is a $\Delta$-inverse of $\omega$.



    Now consider $J\in\mc{X}^{k}\lr{\operatorname{Reg}(C)}$ with
    \eq{
    J= (-1)^{k-2} N\w\mc{J},~~~~N=n_{1}\w...\w n_{k-2}.
    }
    Note that $\iota_{dC}J= (-1)^{k-2}\mc{J}$. Now, for any $Y \in \mc{X}_{\Delta}(\operatorname{Reg}(C))$ it holds from Theorem \ref{thm:classicalFromGeneralised} that $ \iota_Y w = \iota_Y(\omega\w dC) = (\iota_Y \omega)\w dC$. Then, using \eqref{eq:mcJisDeltaInverse},
    \eq{ \iota_{\iota_Y w} J = \iota_{(\iota_Y \omega)\w dC} J = \iota_{(\iota_Y \omega)} \iota_{dC} J = \iota_{(\iota_Y \omega)} \mc{J} = Y.}
    That is, $J$ is a $\Delta$-inverse of $w$.

    Finally, if $X_H$ solves $\iota_{X_H} w = -dH\w dC$, then $X_H \in \mc{X}_\Delta(\operatorname{Reg}(C))$, thus,
    \eq{X_H = \iota_{\iota_{X_H}w}J = -\iota_{dH\w dC} J = - \iota_{dH} \mc{J}. }

\end{proof}

\section{Form flatness,
  Lie-Darboux theorems, and Liouville measures}
  \label{sec:flatnessDarbouxEtc}

In this section we explore how the cornerstone theorems of classical Hamiltonian mechanics are modified in a generalized Hamiltonian theory of degree $k$.
First, we show the following Lie-Darboux type theorem for closed $3$-forms associated to Hamiltonian systems with two independent conserved quantities.

\begin{mydef}
    Let $\omega\in \Omega^2\lr{\M}$ denote a (not necessarily closed) $2$-form and $dC\in \Omega\lr{\M}$ an exact $1$-form on $\M$. Then $\omega$ is said to be of \emph{constant rank with respect to $dC$ in a set $U$} if the $2$-form
    \[ \omega_{C(\bol{x})}|_{\bol{x}} := i_{C(\bol{x})}^* (\omega |_{\bol{x}}) \]
    has the same rank for all $\bol{x} \in U$.
\end{mydef}

\begin{theorem}\label{thm:DarbouxCoords}
    Let $w$ be a closed $3$-form on $\M$ that decomposes as $w = \omega \w dC$ with $\omega\in \Omega^2\lr{\M}$ and $dC\in \Omega\lr{\M}$.
    Then, for every point $\bol{x}\in M$ such that $\omega$ is of constant rank with respect to $dC$ in a neighborhood of $\bol{x}$, there
    exist a neighborhood $U$ of $\bol{x}$ and a coordinate system $\lr{p^1,...,p^\ell,q^1,...,q^\ell,G^1,...,G^\tau}$ with $n=2\ell+\tau$ such that
    \begin{equation}
        w=\omega_0\w dC,~~~~\omega_0=\sum_{i=1}^\ell dp^i\w dq^i~~~~{\rm in}~~U,
    \end{equation}
    where $2\ell$ is the rank of $\tilde{\omega}_{C(\bol{x})} := i_{C(\bol{x})}^* \omega$ at $\bol{x}$.
\end{theorem}

\begin{proof}
    Let $\tilde{U}$ be the neighborhood on which $\omega$ is constant rank $2\ell$ with respect to $dC$. Choosing local coordinates $\lr{x^1,...,x^{n}}=\lr{x^1,...,x^{n-1},C}$ in $\tilde{U}$ of $\bol{x}$, we have
    \begin{equation}
        \omega=\sum_{i<j}\omega_{ij}dx^{ij}=\sum_{i=1}^{n-1}\omega_{in}dx^i\w dC+\sum_{i<j}^{n-1}\omega_{ij}dx^{ij}.\label{omega}
    \end{equation}
    Define ${\omega}'=\sum_{i<j}^{n-1}\omega_{ij}dx^{ij}$ and observe that the rank of $\omega'$ is constant in $\tilde{U}$ if and only if $\omega$ is of constant rank with respect to $dC$ in $\tilde{U}$. By assumption, it follows that $\omega'$ is of constant rank $2\ell$ in $\tilde{U}$.

    Evidently $w={\omega}'\w dC$ in $\tilde{U}$ and, if $\tilde{\omega}_c := i_c^* \omega$ and $\tilde{\omega}_c' := i_c^* \omega'$, then $\tilde{\omega}_c = \tilde{\omega}_c'$. Since $dw=d\omega'\w dC=0$, it follows that $d\tilde{\omega}'_c=i_c^*d\omega'=0$ in any local level set $\Sigma_c \cap \tilde{U}$. Then, by the generalised Lie-Darboux theorem \cite{DeLeon89}, there exists a neighborhood $U \subseteq \tilde{U}$ of $\bol{x}$ and $n-1$ local coordinates $\lr{p^1,...,p^\ell,q^1,...,q^\ell,G^1,...,G^{\tau-1}}$, $\tau = n - 2\ell$ such that
    \begin{equation}
        \tilde{\omega}_c = \tilde{\omega}'_c = \sum_{i=1}^{\ell} dp^i\w dq^i~~~~{\rm in}~~\Sigma_c \cap U.
    \end{equation}
    Moreover, these coordinates depend smoothly on $c$.
    The smooth dependence on $c$ allows the coordinates $p^i,q^i:C^{\infty}\lr{\Sigma_c\cap U}\rightarrow\mathbb{R}$ to also define smooth functions
    $p^i,q^i:C^{\infty}\lr{U}\rightarrow\mathbb{R}$.
    Indeed, they have the form $p^i\lr{x^1,...,x^{n-1},C}$ and
    $q^i\lr{x^1,...,x^{n-1},C}$, $i=1,...,\ell$.
    Finally, observe that
    \eq{
    \   \lr{\omega'-\sum_{i=1}^{\ell}dp^i\w dq^i} \w dC=0~~~~{\rm in}~~U.
    }
    Then,
    \begin{equation}
        w=
        \sum_{i=1}^{\ell}{d}p^i\w{d}q^i\w dC=\omega_0\w dC.
    \end{equation}
\end{proof}

\begin{corollary}\label{cor:localInvariantMeasure}
    Let $w$ be a closed $3$-form on $\M$ that decomposes as $w = \omega \w dC$ with $\omega\in \Omega^2\lr{\M}$ and $dC\in \Omega\lr{\M}$. Assume that $\bol{x} \in M$ is such that $\omega$ is of constant rank with respect to $dC$ in a neighborhood of $\bol{x}$, and let $(p^1,\dots,p^l,q^1,\dots,q^l,G^1\dots,G^\tau)$ be the Darboux coordinates guaranteed from Theorem \ref{thm:DarbouxCoords}. Then, given a $1$-form $dH\in \Omega\lr{\M}$, linearly independent from $dC$ in $U$, the local phase space measure $d\Pi=dp^1\w ...\w dp^{\ell}\w dq^1\w ...\w dq^{\ell}\w dG^1\w...\w dG^\tau$ is an invariant measure in $U$ for the generalized Hamiltonian system  $X\in \mc{X}\lr{U}$ such that
    \begin{equation}
        \iota_Xw=-dH\w dC,\label{sys}
    \end{equation}
    provided that such $X$ exists. In addition,
    \begin{equation}
        \iota_{\tilde{X}}\tilde{\omega}_0=-{d}\tilde{H}~~~~{\rm in}~~\Sigma_c,
    \end{equation}
    where $\tilde{\omega}_0=i_c^*\omega_0$,   $\tilde{H}=i_c^*H$, and $\tilde{X}$ denotes the restriction of $X$ to $\Sigma_c$.
\end{corollary}

\begin{proof}
    Let $X\in \mc{X}\lr{U}$ solve system \eqref{sys}. Recalling that, by hypothesis, $dH$ and $dC$ are linearly independent in $U$ and noting that $0=\iota_X\iota_Xw=-\lr{\iota_XdH} dC+\lr{\iota_XdC} dH$, it follows that $\iota_XdH=\iota_XdC=0$. On the other hand, $\iota_Xw=\iota_X\omega_0\w dC=-{d}{H}\w dC$, which implies
    \begin{equation}
        \iota_{\tilde{X}}\tilde{\omega}_0=-{d}\tilde{H}~~~~{\rm in}~~\Sigma_{c}\cap U.\label{sys2}
    \end{equation}
    Since ${d}\tilde{\omega}_0=0$, equation \eqref{sys2} defines a Hamiltonian system with invariant measure $\lr{\bigwedge_{i=1}^{\ell}{d}p^i\w {d}q^i}\w {d}G^1\w ...\w {d}G^{\tau-1}$, $\tau = n-2\ell$ in $\Sigma_c\cap U$.
    Set $\lr{G^1,...,G^{\tau}}=\lr{G^1,...,G^{\tau-1},C}$. It follows that
    \begin{equation}
        d\Pi=\lr{\bigwedge_{i=1}^{\ell}dp^i\w dq^i}\w dG^1\w...\w dG^{\tau},
    \end{equation}
    defines an invariant measure for $X$ in $U$.
\end{proof}

{Theorem \ref{thm:DarbouxCoords} shows that, under appropriate hypothesis, the $3$-form $w = \omega \w dC$ can be flattened near any point $\bol{x}$ with a neighbourhood that $\omega$ is of constant rank with respect to $dC$ in. This is the essence of the usual Lie-Darboux theorem for symplectic forms. More generally, we have the following definition of a locally flat point $\bol{x}\in \M$.
    \begin{mydef}
        A $k$-form $w$ on a manifold $\M$ is said to be \emph{locally flat near $\bol{x}\in M$} if there exists coordinates $(x^1,\dots,x^n)$ on a neighbourhood $U$ of $\bol{x}$ such that
        \[ w = \sum_{1\leq i_1<\dots<i_k\leq n}A_{i_1\dots i_k}dx^{i_1\dots i_k}  \]
        with $A_{i_1 \dots i_k} \in \R$ constant for all $ 1\leq i_1<\dots<i_k\leq n$.
    \end{mydef}
    In contrast with symplectic geometry, the invertibility of a $k$-form $w$ is not sufficient to guarantee the existence of a local flattening diffeomorphism \cite{Ryvkin,Sato21,Sato24}. The following proposition provides flatness conditions for a class of closed $k$-forms $w$
    by application of Moser's method. This result is then used to construct examples of $3$-form flattening. }

\begin{proposition} \label{prop:flattening}
    Let $w$ be a closed $k$-form on $\M$ and let $(y^1,\dots,y^n)$ be local coordinates in a neighborhood $U$ of a given point $\bol{x}\in\M$. Define $w_0 \in \Omega^k(U)$ and $Z_0 \in \mc{X}(U)$ by
    \eq{w_0 =\sum_{1\leq i_1<...<i_k\leq n}A_{i_1...i_k}dy^{i_1...i_k},\qquad Z_0 = \frac{1}{k} y^i \partial_i, }
    where all $A_{i_1...i_k}\in\mathbb{R}$. Further, suppose that there exists vector fields $X_t, Z \in \mc{X}\lr{\M}$ so that, with appropriate restriction,
    \begin{subequations}
        \begin{align}
             & \mc{L}_Zw_0=w,                                                            \\
             & \lr{\mc{L}_{tX_t}\mc{L}_Z-\mc{L}_{Z_0-Z-\lr{1-t}X_t}}w_0=0,\label{flatc}.
        \end{align}
    \end{subequations}
    If the flow $\Phi_t$ of $X_t$ is well-defined for all $(\bol{x}_0,t) \in U\times [0,1]$ then there exists
    coordinates $\lr{x^1,...,x^n}$ in $U$ such that
    \begin{equation}
        w = \sum_{1\leq i_1<...<i_k\leq n}A_{i_1...i_k}dx^{i_1...i_k} ~~~~{\rm in}~~U.
    \end{equation}
\end{proposition}
\begin{proof}
    We apply Moser's argument to the family of closed $k$-forms $w_t=t w+\lr{1-t}w_0$. We want to construct a family of diffeomorphisms $\Phi_t$ so that $\Phi_t^* w_t = w_0$. If such a family exists, then in particular, $\Phi_1^* w_1 = \Phi_1^*w = w_0$, and $\Phi_1$ is hence the desired coordinate transformation.
    To achieve this, we show the desired transformations $\Phi_t$ is given by the flow of the non-autonomous vector field $X_t$.

    Observe that
    \[  \frac{d}{dt} \Phi_t^* w_t = \Phi_t^*\lr{\partial_t w_t + \mc{L}_{X_t} w_t } = \Phi_t^*\lr{w - w_0 + \mc{L}_{X_t} w_t }  \]
    Using the fact that $w_0=\mc{L}_{Z_0}w_0$ and that by assumption $\mc{L}_Z w_0 = w$, it follows that
    \[ \frac{d}{dt} \Phi_t^* w_t = \Phi_t^*\lr{ \lr{ \mc{L}_{t X_t} \mc{L}_Z - \mc{L}_{Z_0 - Z  - (1-t)X_t} } w_0} = 0,\ \]
    with the last equality holding by assumption. It follows that $\Phi_t^* w_t = \Phi_0^* w_0 = w_0$ whenever $\Phi_t$ is well-defined. By assumption $\Phi_t$ is well defined for $t\in[0,1]$ and all $\bol{x}_0\in U$ and the result follows.
\end{proof}

\begin{remark}
    If a coordinate change mapping a $k$-form $w$ into the flat form $w_0$ has to be found by Moser's method, it is necessary that $w=\mc{L}_Zw_0$ for some $Z\in \mc{X}\lr{\M}$. Indeed, evaluating $\p_tw_t+\mc{L}_{X_t}w_t$ at $t=0$, and recalling that $w_0=\mc{L}_{Z_0}w_0$, one obtains $Z=Z_0-X_0$.
\end{remark}

\begin{remark}
    Consider the setting of Proposition \ref{prop:flattening}.  Denote with $Y\in \mc{X}\lr{\M}$ any Lie symmetry of $w_0$ so that $\mc{L}_Yw_0=0$, and suppose that $\p X_t/\p t=0$. Then, equation \eqref{flatc} leads to the system of equations
    \begin{equation}
        \mc{L}_{X_t}\mc{L}_{Z-Z_0}w_0=0,~~~~\mc{L}_{Z-Z_0-X_t}w_0=0.
    \end{equation}
    Setting $X_t=X$ and $Z=Z_0-X+Y$, one arrives at the equation
    \begin{equation}
        \mc{L}_{X}^2w_0=0.\label{flatc2}
    \end{equation}
    Hence, a flattening coordinate change for the form $w=\mc{L}_{Z}w_0=\lr{1-\mc{L}_X}w_0$ can be constructed by looking at solutions $X$ of equation \eqref{flatc2} and the flow $\Phi_t$.
\end{remark}

\begin{example}
    Consider the $3$-form  $w_0=dx^{123}+dx^{124}$ in $\mathbb{R}^4$. We look for a solution of equation \eqref{flatc2} in the form $X=X^3\p_3+X^4\p_4$. We have
    \begin{equation}
        \begin{split}
            \mc{L}_X^2w_0&=d\iota_X\frac{\p\lr{X^3+X^4}}{\p x^i}dx^{12i}=d\lrc{\lrs{X^3\frac{\p\lr{X^3+X^4}}{\p x^3}+X^4\frac{\p\lr{X^3+X^4}}{\p x^4}}dx^{12}}\\
            &=\frac{\p}{\p x^3}\lrs{X^3\frac{\p\lr{X^3+X^4}}{\p x^3}+X^4\frac{\p\lr{X^3+X^4}}{\p x^4}}dx^{123}+\frac{\p}{\p x^4}\lrs{X^3\frac{\p\lr{X^3+X^4}}{\p x^3}+X^4\frac{\p\lr{X^3+X^4}}{\p x^4}}dx^{124}.
        \end{split}
    \end{equation}
    This quantity vanishes provided that
    \begin{equation}
        X^3\frac{\p\lr{X^3+X^4}}{\p x^3}+X^4\frac{\p\lr{X^3+X^4}}{\p x^4}=f\lr{x^1,x^2}.
    \end{equation}
    An explicit solution can be obtained, for example, by setting $X^3=X^4=\sqrt{\lr{x^3+x^4}f/2}$ for $\lr{x^3+x^4}f\geq 0$. In this case,
    \begin{equation}
        w=\mc{L}_{Z_0-X}w_0=w_0-dx^{12}\w d\sqrt{2\lr{x^3+x^4}f}.
    \end{equation}
    Let us determine Moser's coordinate change. For simplicity, we assume $f\geq 0$. Since $X^1=X^2=0$, we have $x^1=x^1_0$ as well as $x^2=x^2_0$. We must further solve
    \begin{equation}
        \frac{dx^i}{dt}=\sqrt{\frac{\lr{x^3+x^4}f}{2}},~~~~x^i\lr{0}=x^i_0,~~~~i=3,4.
    \end{equation}
    We find
    \begin{equation}
        x^3=\frac{x_0^3-x_0^4}{2}+\frac{1}{2}\lr{\sqrt{x^3_0+x^4_0}+\sqrt{\frac{f}{2}}t}^2,~~~~x^4=\frac{x_0^4-x_0^3}{2}+\frac{1}{2}\lr{\sqrt{x^3_0+x^4_0}+\sqrt{\frac{f}{2}}t}^2.
    \end{equation}
    Using these  expressions, it follows that
    \begin{equation}
        dx^{123}+dx^{124}=dx^{12}\w d\lr{\sqrt{x_0^3+x_0^4}+\sqrt{\frac{f}{2}}t}^2=dx^{123}_0+dx^{124}_0+t dx^{12}_0\w d\sqrt{2\lr{x_0^3+x_0^4}f}.
    \end{equation}
    Noting that $\sqrt{x_0^3+x_0^4}={\sqrt{x^3+x^4}-\sqrt{\frac{f}{2}}t}$, we obtain the desired result
    \begin{equation}
        dx_0^{123}+dx^{124}_0=dx^{123}+dx^{124}-tdx^{12}\w d\sqrt{2\lr{x^3+x^4}f}.
    \end{equation}
\end{example}

\begin{example}
    Consider the $4$-form $w_0=dx^{1234}+dx^{1256}$ in $\mathbb{R}^6$. Let's consider a region $x^3+x^4\geq 0$, $x^5+x^6\geq 0$. Then, the vector field
    \eq{
        X=f\lr{x^1,x^2,x^3-x^4}\sqrt{x^3+x^4}\lr{\p_3+\p_4}+g\lr{x^1,x^2,x^5-x^6}\sqrt{x^5+x^6}\lr{\p_5+\p_6},
    }
    where $f\lr{x^1,x^2,x^3-x^4}$ and $g\lr{x^1,x^2,x^5-x^6}$ are
    arbitrary functions of their arguments, satisfies $\mc{L}_{X}^2w_0=0$. The correponding $4$-form $w=\lr{1-\mc{L}_X}w_0$ is
    \eq{
    w=w_0-\frac{f}{\sqrt{x^3+x^4}}dx^{1234}-\frac{g}{\sqrt{x^5+x^6}}dx^{1256}.
    }
\end{example}

\section{Remarks on the Jacobi identity and invariant measures in generalized Hamiltonian mechanics}
\label{sec:JacobiIdentity}

In this section we explore some of the consequences of the generalization of Hamiltonian mechanics introduced in the previous sections.
In particular, we discuss the formulation of the Jacobi identity and the availability of invariant measures in generalized Hamiltonian theories of degree $k$.

\begin{proposition}
    Let $J\in\mc{X}^k\lr{\M}$ denote a Poisson $k$-tensor. Then, $J$ satisfies the Jacobi indentity in terms of the Poisson bracket
    \eq{
        \lrc{f,g}=\iota_{df}\iota_{dg}\iota_{dC^1}...\iota_{dC^{k-2}}J,
    }
    where $dC^1,...,dC^{k-2}$ are those exact $1$-forms such that $\iota_{dC^1}...\iota_{dC^{k-2}}J=\mc{J}$ is a Poisson $2$-tensor, and $f,g\in C^{\infty}\lr{\M}$.
\end{proposition}

\begin{proof}
    This follows from the definition of Poisson $k$-tensor: the $2$-tensor  $\mc{J}=\iota_{dC^1}...\iota_{dC^{k-2}}J$    is a Poisson $2$-tensor, and thus satisfies the Jacobi identity.
\end{proof}

It is possible to express the Jacobi identity in terms of the components of $J$ in a local coordinate system:
\begin{proposition}
    Let $\lr{x^1,...,x^n}=\lr{x^1,...,x^{n-k+2},C^1,...,C^{k-2}}$ denote a local coordinate system in a neighborhood  $U\subset\M$ and consider a Poisson $k$-tensor $J\in \mc{X}^k\lr{\M}$ such that $\iota_{dC^1}...\iota_{dC^{k-2}}J=\mc{J}$ is a Poisson $2$-tensor. Then, in $U$, the Jacobi identity reads as
    \eq{
    J^{imr...n}J_{m}^{j\ell r...n}+J^{jmr...n}J_{m}^{\ell ir...n}+J^{\ell mr...n}J_{m}^{ijr...n}=0~~~~\forall i,j,\ell=1,...,n,~~~~r=n-k+3.\label{JIloc}
    }
\end{proposition}
\begin{proof}
    Note that
    \eq{
    \iota_{dC^1}...\iota_{dC^{k-2}}J=\sigma \sum_{i<j}J^{ijr...n}\p_{ij},~~~~r=n-k+3,
    }
    where $\sigma$ is a sign factor.
    It is now clear that \eqref{JIloc} is the Jacobi identity for the Poisson $2$-tensor $\sum_{i<j}J^{ijr...n}\p_{ij}$, where $r=n-k+3$.
\end{proof}

\begin{remark}
    Any locally flat $k$-tensor $J\in \mc{X}^k\lr{\M}$  is a Poisson $k$-tensor in any sufficiently small neighborhood $U\subset\M$. Indeed, denoting with $\lr{x^1,...,x^n}$ the local flattening coordinates such that $J^{i_1...i_k}\in\mathbb{R}$, the $2$-tensor $\iota_{dx^{n-k+3}}...\iota_{dx^n}J$
    has constant components, and thus locally defines a Poisson $2$-tensor. 
\end{remark}

Poisson $k$-tensors also has a relation to the fundamental identity encountered in the axiomatic formulation of Nambu(FI) mechanics.
The fundamental identity for the triple bracket $\lrc{\cdot,\cdot,\cdot}:C^{\infty}\lr{\M}\times C^{\infty}\lr{\M}\times C^{\infty}\lr{\M}\rightarrow C^{\infty}\lr{\M}$ describes distribution of time derivatives \cite{Takhtajan1994},
\begin{equation}
    \frac{d}{dt}\lrc{f,g,h}=\lrc{\frac{df}{dt},g,h}+\lrc{f,\frac{dg}{dt},h}+\lrc{f,g,\frac{dh}{dt}},\label{ddt}
\end{equation}
where time evolution is assigned according to $df/dt=\lrc{f,H,C}=J^{ijk}f_iH_jC_k$ with
$f,g,h,C,H\in C^{\infty}\lr{\M}$.

\begin{proposition}
    Let $\lr{x^1,x^2,x^3}=\lr{x^1,x^2,C}$ denote a local coordinate system in a neighborhood $U\subset\M$.
    If $J\in \mc{X}^3\lr{\M}$ satisfies the fundamental identity, then $J$ satisfies the Jacobi identity,
    \begin{equation}
        J^{imn}J^{jkn}_m+J^{jmn}J^{kin}_m+J^{kmn}J^{ijn}_m=0~~~~\forall i,j,k=1,...,n.\label{GJ102}
    \end{equation}
    The converse is not true: a $3$-tensor $J$ satisfying the Jacobi identity is not guaranteed to satisfy the fundamental identity.
\end{proposition}

\begin{proof}
    Explicitly, equation \eqref{ddt} gives the conditions for all $i,j,k,u,v=1,...,n$,
    \begin{subequations}
        \begin{align}
            J^{uvq}J_{u}^{ijk} & =J^{ujk}J_{u}^{ivq}+J^{uki}J_u^{jvq}+J^{uij}J_u^{kvq},\label{FIa} \\
            0                  & =J^{ijk}J^{uvq}+J^{qjk}J^{uiv}+J^{ujk}J^{iqv}.\label{FIb}
        \end{align}
    \end{subequations}
    In equation \eqref{FIa}, set $q=k=n$ to obtain,
    \begin{equation}
        J^{uvn}J_u^{ijn}=J^{ujn}J_u^{ivn}+J^{iun}J_u^{jvn},
    \end{equation}
    which is equivalent to equation \eqref{GJ102}.
    It is however clear that \eqref{GJ102} is not sufficient to fulfill \eqref{FIa} and \eqref{FIb}.
\end{proof}
Section 6.2 below gives an explicit example in which a Poisson $3$-tensor fails to satisfy the fundamental identity.


\begin{proposition}
    Let $J\in\mc{X}^k\lr{\M}$
    denote a Poisson $k$-tensor.
    Assume that the associated Poisson $2$-tensor  $\iota_{dC^1}...\iota_{dC^{k-2}}J$ has constant rank $2m=n-s$ in $\M$. Then, $J$
    admits a local  $\Delta$-inverse $w\in\Omega^3\lr{U}$: 
    for any $p\in \M$ there exists a neighborhood $U$ of $p$ and a local coordinate system $\lr{x^1,...,x^n}=\lr{x^1,...,x^{2m},C^1,...,C^s}$, $s\geq k-2$, such that the local  $\Delta$-inverse $w$ has expression
    \begin{equation}
        w=\omega_c\w dC^1\w ...\w dC^{k-2},~~~~\omega_c=\sum_{i=1}^m dx^i\w dx^{m+i}~~~~{\rm in}~~U,
    \end{equation}
    with $\Delta=\lrc{X\in\mc{X}\lr{U}:\iota_XdC=0,dC=dC^1\w...\w dC^s}$.
\end{proposition}

\begin{proof}
    Since $J$ is a Poisson $k$-tensor, we may choose coordinates  $\lr{x^1,...,x^n}=\lr{x^1,...,x^{2m},C^1,...,C^s}$ in a sufficiently small neighborhood $U$ of any $p\in \M$.
    Here, $C^1,...,C^{k-2}$ are those functions such that $\iota_{dC^1}...\iota_{dC^{k-2}}J$ is a Poisson $2$-tensor, while $dC^{k-1},...,dC^{s}$ are those additional  exact $1$-forms belonging to the kernel of $\iota_{dC^1}...\iota_{dC^{k-2}}J$.
    Since $\iota_{dC^1}...\iota_{dC^{k-2}}J$ is a Poisson $2$-tensor, the local coordinates can be chosen so that $\iota_{dC^1}...\iota_{dC^{k-2}}J=\lr{-1}^{k-1}\sum_{i=1}^m\p_{i}\w \p_{m+i}$ \cite{Littlejohn82}.
    It follows that $w=\omega_c\w dC^1\w...\w dC^{k-2}$ with $\omega_c=\sum_{i=1}^mdx^i\w dx^{m+i}$ is the local $\Delta$-inverse of $J$. Indeed, taking $X\in \Delta$,
    \eq{
    \iota_{\iota_Xw}J=&J\lr{\sum_{i=1}^m \lr{X^idx^{m+i}-X^{m+i}dx^i}\w dC^{1}\w...\w dC^{k-2}}\\=&
    \sum_{i=1}^m\sum_{j_1<...<j_{k}}J^{j_1...j_k}X^i\langle dx^{m+i}\w dC^1\w...\w dC^{k-2},\p_{j_1...j_{k-1}}\rangle\p_{j_k}+...
    \\=&
    \lr{-1}^{k-1}
    \sum_{i=1}^m\sum_{j_1<...<j_{k}}J^{j_1...j_k}X^i\langle dx^{m+i}\w dC^1\w...\w dC^{k-2},\p_{j_2...j_k}\rangle\p_{j_1}+...\\=&\sum_{i=1}^m
    \lr{-1}^{k-1}
    J^{i\,m+i\,n-s+1...n-s+1+k-2}X^i\p_{i}+...=\sum_{i=1}^m\lr{X^i\p_i+X^{m+i}\p_{m+i}}.
    }
\end{proof}

We conclude this section with some remarks on the availability of invariant measures in a generalized Hamiltonian theory of degree $k$.

\begin{mydef}\label{def:volumePreservation}
    Let $d\mu$ be a smooth volume form on $M$.
    Consider the vector field $X=-\iota_{dH^1}...\iota_{dH^{k-1}}J$, with
    $J\in\mc{X}^k\lr{M}$
    and $H^1,...,H^{k-1}\in C^{\infty}\lr{M}$.
    We say that $J$
    is measure preserving whenever there exists a non-vanishing function $g\in C^{\infty}\lr{M}$ such that
    \eq{
    \mc{L}_X \lr{g d\mu}=0,~~~~\forall H^1,...,H^{k-1}\in C^{\infty}\lr{M}.
    }
\end{mydef}

\begin{remark}
    In coordinates $\lr{x^1,...,x^n}$ and with $d\mu=dx^1\w...\w dx^n$,  Definition  \ref{def:volumePreservation} corresponds to the condition
    \begin{equation}
        \frac{\p}{\p x^{i}}\lr{gJ^{ij_1...j_{k-1}}}=0,~~~~\forall j_1,...,j_{k-1}=1,...,n,\label{mesp}
    \end{equation}
    which ensures $gdx^1\w...\w dx^n$ is an invariant measure for all choices of the Hamiltonians $H^1,...,H^{k-1}$.
\end{remark}

Even if \eqref{mesp} is not satisfied, an invariant measure may be recovered by adding a new dimension
with coordinate $x^{n+1}$. The following proposition verifies this for $k=3$ and $\M \simeq \R^n$.

\begin{proposition}
    Take $J\in\mc{X}^3\lr{\M}$. 
    Let $\lr{x^1,...,x^n}$ denote local coordinates in some neighborhood $U\subset\M$.
    Then, the
    $n+1$-dimensional antisymmetric contravariant $3$-tensor
    $\mf{J}\in\mc{X}^3\lr{U\cp\mathbb{R}}$ given by
    \eq{
    \mf{J}=J-\sum_{i<j<k}x^{n+1}\lr{\frac{\p J^{mij}}{\p x^m}\delta_{n+1}^k+\frac{\p J^{mjk}}{\p x^m}\delta_{n+1}^i+\frac{\p J^{mki}}{\p x^m}\delta_{n+1}^j}\p_i\w\p_j\w\p_{k},\\
    }
    is measure preserving in $U$. 
    Furthermore, the equations of motion $dx^i/dt=J^{ijk}H_jC_k$, $i=1,...,n$, with the Hamiltonians $H$ and $C$ independent of $x^{n+1}$, remain unchanged in $U$, that is
    \sys{
    &\frac{dx^i}{dt}=\mf{J}^{ijk}H_jC_k=J^{ijk}H_jC_k,~~~~i=1,...,n,\\
    &\frac{dx^{n+1}}{dt}=\mf{J}^{n+1jk}H_jC_k=-x^{n+1}\sum_{i=1}^n\p_i X^i,
    }{mesp2}
    where, as usual, $X=d\bol{x}/dt$.
\end{proposition}
\begin{proof}
    The proof of this statement can be obtained by noting that $\sum_{i=1}^{n+1}\p_i\mf{J}^{ijk}=0$ and by evaluating  $dx^i/dt=\mf{J}^{ijk}H_jC_k$  for $i=1,...,n+1$.
\end{proof}

\section{Examples} \label{sec:examples}

In this section, three examples are provided of generalized Hamiltonian systems of degree $k=3$.
In the first example, it is shown that quasisymmetric magnetic fields can be written as a Nambu system.  Quasisymmetric magnetic fields play a pivotal role in the design of next-generation fusion reactors known as stellarators \cite{Rod,Hel,Burby}.
This novel Nambu system provides a novel pathway to study quasisymmetric magnetic fields. Furthermore, this formulation
highlights previously unnoticed properties of quasisymmetric magnetic fields, such as the existence of an invariant measure for the quasisymmetry.
The aim of the second example is to demonstrate the existence of a classical Hamiltonian system with two invariants whose
$3$-tensor $J$ fails to satisfy the fundamental identity of Nambu(FI) mechanics. However, the system still admits a description as a generalized Hamiltonian sytem of degree $k=3$ in accordance with Theorem \ref{thm:wFromOmega}.
Finally, the third example concerns a $4$-dimensional dynamical system
endowed with an antisymmetric contravariant $2$-tensor $\mc{J}$ failing to satisfy the Jacobi identity, and yet possessing a generalized Hamiltonian structure.
It is shown that the hypothesis of Theorem \ref{thm:classicalFromGeneralised} are verified by this system. Consequently, classical Hamiltonian systems can be recovered on some $2$-dimensional submanifolds of the original $4$-dimensional domain.

\subsection{Quasisymmetry as a generalized Hamiltonian system}

Given a magnetic field $\bol{B}\in \mc{X}\lr{\M}$ in some region $\M\subset\mathbb{R}^3$ and a function $\Psi\in C^{\infty}\lr{\M}$, a quasisymmetry $\bol{u}\in \mc{X}\lr{\M}$ of $\bol{B}$ is any solution to
\begin{equation}
    \nabla\cdot\bol{B}=0,~~~~\nabla\cdot\bol{u}=0,~~~~\bol{B}\cp\bol{u}=\nabla\Psi,~~~~\bol{u}\cdot\nabla {B}=0,
\end{equation}
where $B=\abs{\bol{B}}$.
Consider a magnetic field $\bol{B}$ such that, on a region $U\subset\M$, it holds that $\nabla\Psi\cp\nabla B\neq 0$ as well as $\bol{B}\cdot\nabla B\neq 0$. Then, in $U$, it can be computed that
\begin{equation}
    \bol{u}=\frac{\nabla\Psi\cp\nabla B}{\bol{B}\cdot\nabla B},~~~~\bol{B}\cdot\nabla B=f\lr{\Psi,B},~~~~\bol{B}\cdot\nabla\Psi=0, \label{QS2}
\end{equation}
for some $f\lr{\Psi,B} \neq 0$. The quasisymmetry $\bol{u}$ can be regarded as a generalized Hamiltonian system on $U$. Indeed, consider the $3$-form $w$ given by the closed form
\begin{equation}
    w=-\lr{\bol{B}\cdot\nabla B} d \mu,
\end{equation}
where $d \mu$ is the standard Euclidean volume form on $U$.
Then
\begin{equation}
    \iota_{\bol{u}} w =  -\iota_{\nabla\Psi \times \nabla B}d\mu = -d\Psi \w dB.
\end{equation}
Hence $\bol{u}$ is a Hamiltonian vector field for the closed $3$-form $w$ with Hamiltonians $\Psi, B$. Note that, as $w$ is a top form on $U$ and $\mc{L}_{\bol{u}} w = 0$, that $w$ is an invariant measure for $\bol{u}$.

In any local chart on $U$ with standard Euclidean coordinates, we have the associated differential equation for $\bol{u}$ given by
\begin{equation}
    \frac{d\bol{x}}{dt}=\bol{u}=\frac{\epsilon^{ijk}}{\bol{B}\cdot\nabla B}\frac{\p\Psi}{\p x^j}\frac{\p B}{\p x^k}\p_i,
\end{equation}
with Hamiltonians $\Psi$ and $B$ and a Poisson $3$-tensor
\begin{equation}
    J^{ijk}=\frac{\epsilon^{ijk}}{\bol{B}\cdot\nabla B}.
\end{equation}

\subsection{Semiclassical quantum oscillators}

Consider the following dynamical system on $\R^6$ from \cite{Horikoshi2021,Horikoshi2023} that models a $1$-dimensional system of two quantum oscillators. The dynamical variables are $p_i$, $q_i$, $\xi_i$, $i=1,2$, and physically correspond to the expectation values of momentum, position, and squared momentum of the $i^{\text{th}}$ oscillator, respectively. The equations of motion are given by
\begin{subequations}
    \begin{align}
        \dot{p}_1   & = -q_1-\lambda\xi_2,     \\
        \dot{q}_1   & = p_1,                   \\
        \dot{\xi}_1 & = 2q_1 p_1,              \\
        \dot{p}_2   & = -q_2-2\lambda q_1 q_2, \\
        \dot{q}_2   & = p_2,                   \\
        \dot{\xi}_2 & = 2q_2 p_2.
    \end{align}\label{Hori}
\end{subequations}
One can verify that the quantities
\begin{equation}
    G_1=\xi_1-q_1^2,\quad G_2=\xi_2-q_2^2,\quad  H=\frac{1}{2}\lr{p_1^2+p_2^2+\xi_1+\xi_2}+\lambda q_1\xi_2
\end{equation}
are independent constants of motion, with the third quantity $H$ corresponding to the energy of the system. The dynamical system \eqref{Hori} can be expressed as a generalized Hamiltonian system. Indeed, defining  
\eq{ J =\p_{p_1}\w\p_{q_1}\w\p_{\xi_1}+\p_{p_2}\w\p_{q_2}\w\p_{\xi_2},\qquad \{f,g,h\} := -\iota_{df}\iota_dg\iota_dh J }
for any $f,g,h\in C^\infty(\R^6)$, and taking $G = G_1 + G_2$, we have 
\eq{ \frac{d}{dt} F = \{ F,G,H \} }
for any observable $F \in C^\infty(\R^6)$. In particular, for $i=1,2$,
\eq{ \dot{p}_i = \{p_i, G,H\},\quad \dot{q}_i = \{q_i,G,H\},\quad \dot{\xi}_i = \{\xi_i,G,H\}.}

As shown in \cite{Sato21}, the $k$-vector $J$ does not satisfy the fundamental identity.
On the other hand, performing the change of variables
$\lr{p_1,q_1,\xi_1,p_2,q_2,\xi_2}\rightarrow\lr{p_1,q_1,G_1,p_2,q_2,G_2}$, the Hamiltonian function becomes,
\begin{equation}
    \tilde{H}=\frac{1}{2}\lr{p_1^2+p_2^2+q_1^2+q_2^2}+\lambda q_1\lr{G_2+q_2^2},
\end{equation}
while system \eqref{Hori} takes the form
\eq{ \dot{p}_i = -\frac{\p \tilde{H}}{\p q_i},\quad \dot{q}_i = \frac{\p \tilde{H}}{\p p_i},\quad \dot{G}_i = 0, \label{Hori3} }
for $i=1,2$.
Equation \eqref{Hori3} shows that system \eqref{Hori} is a noncanonical Hamiltonian system with Casimir invariants $G_1$ and $G_2$. By application of Theorem \ref{thm:wFromOmega}, the system is endowed with a generalized Hamiltonian structure with symplectic $3$-form
\begin{equation}
    w=\omega\w dG_2=\lr{dp_1\w dq_1+dp_2\w dq_1}\w dG_2.
\end{equation}
In addition, the system also admits the invariant measure $dp_1\w dq_1\w dp_2\w dq_2\w dG_1\w dG_2$, in ageement with Corollary \ref{cor:localInvariantMeasure}.

\subsection{An example of application of {Theorem \ref{thm:classicalFromGeneralised}}}

Let $\lr{x^1,...,x^4}$ denote standard Euclidean coordinates on $\mathbb{R}^4$ and consider the domain $\M=\{x^4>0\}$. In $\M$, we define the $2$-vector field $\mc{J}$ as
\begin{equation}
    \mc{J}=\frac{1}{x^4}\lr{\p_2\w\p_1+\p_4\w\p_3}
\end{equation}
The Jacobi identity is not satisfied by this tensor, as one can verify by computing the term
\eq{
    \mc{J}^{3m}\mc{J}_{m}^{12}+\mc{J}^{1m}\mc{J}_{m}^{23}+\mc{J}^{2m}\mc{J}_{m}^{31}=-\frac{1}{\lr{x^4}^3}\neq 0.
}
Let $H\in C^{\infty}(\M)$ be any function satisfying $\frac{\partial H}{\partial x^3}=0$. Then define
\eq{
    X=-\iota_{dH}\mc{J}=\frac{1}{x^4}\lr{\frac{\p H}{\p x^1}\p_2- \frac{\p H}{\p x^2}\p_1- \frac{\p H}{\p x^4}\p_3}.
}
Next, observe that the $2$-form
\eq{
    \omega=x^4\lr{dx^1\w dx^2+dx^3\w dx^4},
}
satisfies
\eq{
    \iota_X\omega=-dH,~~~~d\omega=dx^4\w dx^1\w dx^2\neq 0.
}
On the other hand, the $3$-form $w = \omega \w dx^4$ 
satisfies
\eq{
    \iota_Xw=-dH\w dx^4,~~~~dw=d\omega\w dx^4=0.
}
Hence, this $4$-dimensional  dynamical system does not possess a classical Hamiltonian structure, and yet it can be cast as a generalized Hamiltonian system. 

Next, let us further assume that \( \frac{\partial H}{\partial x^4} = 0 \). Under this assumption, the 
vector field 
\eq{
X = \frac{1}{x^4} \left( \frac{\partial H}{\partial x^1} \partial_2 - \frac{\partial H}{\partial x^2} \partial_1 \right), 
}
satisfies the relation:
\eq{
\iota_X \omega = -dH.  
}
On the other hand, the \( 4 \)-form \( w = \omega \wedge dx^3 \wedge dx^4 \) satisfies:
\eq{
\iota_X w = -dH \wedge dx^3 \wedge dx^4, \quad dw = d\omega \wedge dx^3 \wedge dx^4 = 0.
}
Furthermore, the hypotheses of Theorem \ref{thm:classicalFromGeneralised} are now verified for \( w \), where
\eq{
X_H = \frac{1}{x^4} \left( \frac{\partial H}{\partial x^1} \partial_2 - \frac{\partial H}{\partial x^2} \partial_1 \right),
}
for any \( H \in C^{\infty}(M) \). This implies that we expect to find a classical Hamiltonian system on the submanifolds defined by:
\eq{
\Sigma_c = \left\{ \bol{x} \in M : x^3 = c^3 \in \mathbb{R}, \, x^4 = c^4 \in \mathbb{R}_{\neq 0} \right\}.
}
One can verify that the pullbacks \( \tilde{\omega} = \iota^*_c \omega \) and \( \tilde{H} = \iota^*_c H \), along with the restriction \( \tilde{X}_H \) of \( X_H \) to \( \Sigma_c \), satisfy:
\eq{
\iota_{\tilde{X}} \tilde{\omega} = -d\tilde{H}, \quad d\tilde{\omega} = 0, \quad \tilde{\omega} = c^4 dx^1 \wedge dx^2.
}

\section{Concluding remarks}
\label{sec:ConcludingRemarks}

In this work,
we explored the connection between
generalized Hamiltonian mechanics (intended as the ideal dynamics of
systems endowed with a phase space conservation law embodied by a closed differential form and a matter conservation law encoded in multiple Hamiltonians) and multisyplectic geometry (the geometry of closed differential forms).
The key of the construction is the notion of invertibility of differential forms, which can be conjugated into (strong) invertibility \eqref{sinv} and weak invertibility \eqref{winv}.
When $k>2$, (strong) invertibility
of a differential form of degree $k$
is no longer sufficient to ensure the existence of a solution $X$ to the equation $\iota_Xw=-\sigma$, where $X$ is a vector field and $\sigma$ a given differential form of degree $k-1$. Indeed,
the space of differential $k-1$-forms has dimension $\binom{n}{k-1}$, which is greater than the dimension of the tangent space whenever $\binom{n}{k-1}>n$ (when $k=2$, this gives $n>3$).
This subtlety is the reason why
a straightforward connection between
the closed $k$-form $w$ and
its inverse antisymmetric $k$-tensor $J$ cannot be established, and Moser's method cannot be generally applied to find a local coordinate change mapping $w$ to a flat (constant) $k$-form.
The correspondence between generalized Hamiltonian mechanics and multisyplectic geometry can however be recovered
by properly taking into account the multiple invariants that characterize generalized dynamics. In addition, a precise correspondence
between generalized Hamiltonian mechanics and classical Hamiltonian mechanics can be obtained, as described in section 3.
In particular, we find that classical Hamiltonian mechanics represents a subset of generalized Hamiltonian mechanics, in contrast with  the formulation of Nambu(FI) mechanics based  on the fundamental identity.

\section*{Acknowledgment}
N.S. would like to thank Z. Yoshida and P. J. Morrison for useful discussion.

\section*{Statements and declarations}

\subsection*{Data availability}
Data sharing not applicable to this article as no datasets were generated or analysed during the current study.

\subsection*{Funding}
The research of NS was partially supported by JSPS KAKENHI Grant No. 21K13851, No. 22H04936, and No. 24K00615. This work was partly supported by MEXT Promotion of Distinctive Joint Research Center Program JPMXP0723833165.

\subsection*{Competing interests}
The authors have no competing interests to declare that are relevant to the content of this article.

\end{document}